\documentclass[10pt,a4paper]{amsart}
\usepackage{mathrsfs}
\usepackage{amsmath}
\usepackage{amssymb}
\usepackage[all]{xy}
\usepackage{latexsym}

\usepackage{datetime}
\usepackage{pgf,tikz}
\usetikzlibrary{matrix,arrows, calc}
\usepackage{graphicx}

\usepackage[latin1]{inputenc}
\usepackage[T1]{fontenc}

\newcommand{ \inj}{ \hookrightarrow}

\DeclareMathOperator{\Hom}{Hom}

\DeclareMathOperator{\ext}{Ext}
\DeclareMathOperator{\Pic}{Pic}

\newcommand{\Z}{\mathbb{Z}}
\newcommand{\F}{\mathbb{F}_2}
\newcommand{\M}{\mathbb{M}_2}

\newcommand{\C}{\mathbb{C}}

\newcommand{\E}{\mathcal{E}}

\newcommand{\st}{\mathrm{Stab}}

\newcommand{\eq}{\cong}
\newcommand{\steq}{\simeq}

\newcommand{\A}{\mathcal{A}}
 
\newcommand{\At}{A/\tau}

\newcommand{\Amod}{{ }_{A}\mathbf{Mod}}
\newcommand{\Amodf}{\Amod^{\textnormal{f}}} 			
 
\newcommand{\Atmod}{{ }_{A/\tau}\mathbf{Mod}} 

\newcommand{\modmod}{/ \hspace{-0.07cm}/}

\theoremstyle{definition}
\newtheorem{de}{Definition}[section]

\theoremstyle{plain}
\newtheorem{thm}[de]{Theorem}
\newtheorem{lemma}[de]{Lemma}
\newtheorem{pro}[de]{Proposition}
\newtheorem{cor}[de]{Corollary}

\newtheorem*{thm*}{Theorem}
\newtheorem*{lemma*}{Lemma}
\newtheorem*{pro*}{Proposition}
\newtheorem*{cor*}{Corollary}

\theoremstyle{remark}
\newtheorem{rk}[de]{Remark}
\newtheorem{ex}[de]{Example}

\newcommand{\quotient}[2]{{\raisebox{.2em}{$#1$}\left/\raisebox{-.2em}{$#2$}\right.}}

\usepackage{chemarrow} 
\usepackage{xspace} 
\usepackage{stmaryrd}
\newcommand{\lto}{\ensuremath{ \mathrel{ \mkern1.5mu\textrm{\arro\symbol{71}} \mkern-2.5mu\textrm{\arro\symbol{71}} \mkern-1.1mu\textrm{\arro\symbol{65}} \mkern+1mu } \xspace }}	
\newcommand{\ltooo}{\ensuremath{ \mathrel{ \mkern1.5mu\textrm{\arro\symbol{71}} \mkern-2.5mu\textrm{\arro\symbol{71}} \mkern-2.5mu\textrm{\arro\symbol{71}} \mkern-2.5mu\textrm{\arro\symbol{71}} \mkern-1.1mu\textrm{\arro\symbol{65}} \mkern+1mu } \xspace }}	
\newcommand{\ltoback}{\ensuremath{ \mathrel{ \mkern+3mu\textrm{\arro\symbol{71}} \mkern-13mu\textrm{\arro\symbol{68}}  \mkern-0.5mu\textrm{\arro\symbol{71}} \mkern-4.5mu\textrm{\arro\symbol{71}} \mkern+4mu } \xspace }}	
\newcommand{\adj}{ \ensuremath{ \mathrel{ \mkern+1.1mu \raise2.45pt\vbox{ \hbox{\lto}} \mkern-29mu  \raise+1.1pt\vbox{ \hbox{ \scalebox{0.45}{$\perp$} }} \mskip-29mu\raise-2.45pt\hbox{\ltoback} \mkern+2mu } \xspace } } 
\newcommand{\lmapsto}{\ensuremath{\mathrel{ \mkern+3.4mu \mapsfromchar \mkern-6mu\textrm{\arro\symbol{71}} \mkern-1.1mu\textrm{\arro\symbol{71}} \mkern-4mu\textrm{\arro\symbol{65}} } \mkern+1mu} \xspace }	
\newcommand{\surj}{\ensuremath{\mathrel{ \mkern1.5mu\textrm{\arro\symbol{71}} \mkern-1.1mu\textrm{\arro\symbol{71}} \mkern-8mu\textrm{\arro\symbol{65}} \mkern-11mu\textrm{\arro\symbol{65}} } \mkern+1mu } \xspace }	
\mathchardef\dashmod="2D

\DeclareMathOperator{\Sq}{Sq}
\newcommand{\cl}{\mathrm{cl}}
\newcommand{\op}{\mathrm{op}}
\newcommand{\ko}{\mathit{ko}}

\title{The Picard group of motivic $\A(1)$}
\author{Bogdan Gheorghe}
\email{gheorghebg@wayne.edu}
\author{Daniel C. Isaksen}
\email{isaksen@wayne.edu}
\author{Nicolas Ricka}
\email{ricka@wayne.edu}
\address{Department of Mathematics, Wayne State University \\
 Detroit, MI 48202}


\begin{document}

\begin{abstract}
We show that the Picard group $\Pic(\A(1))$
of the stable category of modules over $\C$-motivic $\A(1)$
is isomorphic to $\Z^4$.
By comparison, the Picard group of classical
$\A(1)$ is $\Z^2 \oplus \Z/2$.
One extra copy of $\Z$ arises from the motivic bigrading.
The joker is a well-known exotic element of order $2$ 
in the Picard group of 
classical $\A(1)$. 
The $\C$-motivic joker has infinite order.
\end{abstract}

\maketitle

\section{Introduction}

\subsection{The Picard group of $\A(1)$}

Let $\A(1)^{\cl}$ be the subalgebra of the classical mod 2 Steenrod algebra 
generated by $\Sq^1$ and $\Sq^2$.
The stable module category $\st(\A(1)^{\cl})$
is the category whose objects are the finitely generated graded
left $\A(1)^{\cl}$-modules, and whose morphisms are the usual
$\A(1)^{\cl}$-module maps, modulo maps that factor through 
projective $\A(1)^{\cl}$-modules.

The stable module category $\st(\A(1)^{\cl})$ is equipped with a tensor
product over $\F$.  The unit of this pairing is $\F$, and
an object $M$ of $\st(\A(1)^{\cl})$ is invertible if there exists 
another $\A(1)^{\cl}$-module $N$ such that 
$M \otimes_{\F} N$ is stably isomorphic to $\F$.
The Picard group $\Pic(\A(1)^{\cl})$ is the set of invertible
stable isomorphism classes, with group operation given by tensor
product over $\F$.

$\ext$ groups over $\A(1)^{\cl}$ are invariants
of stable isomorphism classes of $\A(1)^{\cl}$-modules.
Thus, $\st(\A(1)^{\cl})$ is the natural category on which
$\ext$ groups over $\A(1)^{\cl}$ are defined.
These $\ext$ groups are of topological 
interest because of the Adams spectral sequence
$$ E_2 = \ext_{\A(1)^{\cl}}( H\F^{*}(X), \F) \Rightarrow \ko_{*}(X)^{\wedge}_2$$ 
converging to $2$-completed $\ko$-homology.

Adams and Priddy computed $\Pic(\A(1)^{\cl})$ while studying infinite loop
space structures on the classifying space $BSO$ 
\cite[Section 3]{AP76}.
They found that the Picard group is isomorphic to $\Z^2 \oplus \Z/2$.
One copy of $\Z$ comes from the grading; one can shift the grading
on $A(1)^{\cl}$-modules to obtain ``new'' $\A(1)^{\cl}$-modules.
The other copy of $\Z$ comes from the algebraic loop functor
that is a formal part of the stable module category;
see Definition \ref{de:omega} below for more details.

The copy of $\Z/2$ in $\Pic(\A(1)^{\cl})$ is the most interesting part
of the calculation.  It is exotic in the sense that it doesn't follow
from the formal theory of stable module categories and Picard groups.  The
copy of $\Z/2$ is generated by the joker $J$ shown in 
Figure \ref{fig:joker}.  It turns out that $J \otimes_{\F} J$ is stably
isomorphic to $\F$, so $J$ has order $2$ in $\Pic(\A(1)^{\cl})$.

\subsection{The motivic setting}

There has been much recent work on the computational side of 
motivic homotopy theory.  In particular, the algebraic properties
of the motivic Steenrod algebra have come under close scrutiny.
As part of this program, it is natural to ask about the Picard
group of the motivic version of $\A(1)$.
The goal of this article is to carry out this computation for
$\C$-motivic $\A(1)$, which is the simplest motivic case.

The fundamental difficulty in the motivic situation is that the
ground ring $\M$ is not a field.  Rather, it is a graded polynomial
ring $\F[\tau]$.  Therefore, we must be careful to insert
$\M$-freeness hypotheses at the appropriate places.

We will show that $\Pic( \A(1) )$ is isomorphic to 
$\Z^4$.  Two copies of $\Z$ arise from the motivic bigrading,
and one copy of $\Z$ comes from the algebraic loop functor.
This leaves one copy of $\Z$, which is generated by the
motivic joker $J$ (see Figure \ref{fig:joker}).  It turns out that
the motivic joker has infinite order.  The order of the motivic
joker is the essential new aspect of the motivic calculation.

There are two main ideas in the proof.
First,
the Hopf algebra $\A(1)/\tau$ is isomorphic to the group algebra
of the dihedral group $D_8$ of order $8$, so $\A(1)/\tau$ is well-understood.
In particular, the Picard group of $\A(1)/\tau$ is known.

Second,
consider the functor that takes an $\A(1)$-module
$M$ to its quotient $M/\tau$.  In general, quotienting is not an exact
functor.  However, it turns out to be well-behaved for $\A(1)$-modules
that are $\M$-free.  Using this well-behaved functor, we can pull back information about the Picard group of $\A(1)/\tau$ to information
about the Picard group of $\A(1)$.

The difference between the $\C$-motivic and classical Picard groups
is a familiar one.  Frequently, motivic computations are larger than
classical ones.  However, they are also often more regular.  This
situation is clearly displayed in our work, where the motivic Picard
group is free, while the classical Picard group has torsion.

We do not consider the Picard group of motivic $\A(1)$
over other base fields.  The $\C$-motivic phenomena described in this
paper will occur over other base fields, but it is possible that additional
complications arise.

Our computation of the Picard group of motivic $\A(1)$ is potentially
useful for the following problem.  From our perspective, the most
essential property of the $\C$-motivic spectrum $\ko$
is that its cohomology is isomorphic to $\A\modmod\A(1)$ \cite{Is10}.
One might ask whether such a $\C$-motivic spectrum is unique.
Suppose that $X$ and $Y$ 
are $\C$-motivic spectra whose cohomology modules are both isomorphic
to $\A\modmod\A(1)$.  In order to construct an equivalence between $X$ and $Y$,
one could compute the maps between $X$ and $Y$ via
the motivic Adams spectral sequence, whose $E_2$-page takes the form
$\ext_{\A} (\A\modmod\A(1), \A\modmod\A(1) )$.  By a standard change of rings
theorem, this $E_2$-page is equal to
$\ext_{\A(1)} (\M, \A\modmod\A(1) )$.  
It is possible that this Adams spectral sequence is analyzable,
because $\A\modmod\A(1)$ probably splits as an $\A(1)$-module into
summands that belong to the Picard group.
We leave the details for future work.


\section{Stable module theory of finite motivic Hopf algebras}

\subsection{Finite motivic Hopf algebras}

We use the same notation and framework as in \cite{Is14}.
We work in the $\C$-motivic setting at the prime 2. The base ring is the motivic cohomology 
$H^{\ast,\ast}( S^{0,0}; \F) $
of the sphere spectrum.
We write $\M$ for this ring; it is isomorphic to 
$\F[\tau]$ with $\tau$ in bidegree $(0,1)$. 
Objects are bigraded in the form $(s,w)$, 
where $s$ corresponds to the classical internal degree and $w$ is the motivic weight. 

Let $\A$ be the $\C$-motivic Steenrod algebra at the prime 2. This Hopf algebra over $\M$ was first computed in \cite{Voe03}, and its structure is thoroughly understood. 

A fundamental difference between the $\C$-motivic and the classical
situations is that the base ring $\M$ is not a field.  
Therefore, we must add freeness over $\M$ as a hypothesis
in Definition \ref{defn:mot-Hopf} below.

\begin{de}
\label{defn:mot-Hopf}
A \emph{finite motivic Hopf algebra} is a cocommutative bigraded Hopf algebra over $\M$ that is finitely generated and free as an $\M$-module.
\end{de}

\begin{ex}
Recall the subalgebras $\A(n)^{\cl}$ and $\E(n)^{\cl}$ of the classical
Steenrod algebra \cite{AM74}.  These subalgebras have $\C$-motivic analogues,
and they are finite motivic Hopf algebras.
\end{ex}

Throughout the article,
$A$ will represent an arbitrary finite motivic Hopf algebra,
while $\A$ represents the $\C$-motivic Steenrod algebra.
Note that $\A$ is not finitely generated as an $\M$-module.
However, we are primarily interested in the subalgebra $\A(1)$ of
$\A$ generated by $\Sq^1$ and $\Sq^2$, and $\A(1)$ is
a finitely generated $\M$-module.

\begin{lemma}
\label{lem:Aproj-Mfree}
Suppose that $A$ is a finite motivic Hopf algebra.
If $M$ is a finitely generated projective $A$-module, then
it is a finitely generated free $\M$-module.
\end{lemma}

\begin{proof}
Suppose that $M$ is a finitely generated projective $A$-module.
Then $M$ is a summand of free $A$-module $F$.  The module $F$ 
is free and finitely generated as an $\M$-module,
since $A$ is free and finitely generated as an $\M$-module.
Therefore, as an $\M$-module, $M$ is a summand of a free $\M$-module.
This shows that $M$ is a finitely generated projective $\M$-module.

It remains to show that finitely generated projective $\M$-modules
are free.
The ring $\M$ is a graded principal ideal domain whose graded ideals are
of the form $(\tau^k)$.
Therefore, 
a finitely generated $\M$-module is a direct sum of a free module and 
cyclic modules of the form $\M/ \tau^k$.
It follows that finitely generated projective $\M$-modules 
are the same as finitely generated free $\M$-modules.
\end{proof}

\subsection{The stable category}

We now recall the basic framework of stable module categories,
as applied to a finite motivic Hopf algebra $A$. The stable category of modules over a group algebra is a classical construction in group representation 
theory \cite[Section 2.6]{CTV03}. In the case of a finite
motivic Hopf algebra $A$, the theory is similar to the case when $A$ is a finite dimensional graded connected Hopf algebra over a field, for which a good reference is
\cite[Section 14.1]{Mar83}. However, since the base ring $\M$ of a finite motivic Hopf algebra is not a field, one has to pay attention to the underlying theory of $\M$-modules, and add $\M$-freeness hypotheses when appropriate.

\begin{de}
\label{defn:Amod}
Let $\Amod$ be the category of bigraded finitely generated left $A$-modules, and let $\Amodf$ be the full subcategory of 
$\Amod$ consisting of left $A$-modules that are free over $\M$.
\end{de}

\begin{de}
\label{defn:Stab}
Let $\st(A)$ be the category whose objects are the same as in $\Amodf$, and whose morphisms are given by
$$\Hom_{\st(A)}(M,N) = \quotient{\Hom_{A}(M,N)}{ \sim},$$ 
where two morphisms $f$ and $g$ are equivalent if 
their difference factors through a projective $A$-module. 
\end{de}

If $M$ and $N$ are objects of $\Amodf$, then we write
$M \eq N$ if $M$ and $N$ are stably equivalent,
i.e., if they are isomorphic in the stable category $\st(A)$.

Our main interest is the Picard group $\Pic(A)$ 
of the stable category of some finite motivic Hopf algebra $A$. 
We will see below in Remark \ref{rk:Mfree} that 
all representatives of
every element in $\Pic(A)$ are actually free over $\M$ 
and thus captured by $\st(A)$. 
In other words, the assumptions about $\M$-freeness in
Definitions \ref{defn:Amod} and \ref{defn:Stab} are no loss
of generality.

In the same vein, it is essential that we use constructions
that preserve $\M$-free $A$-modules.
For example for any finitely generated 
$A$-module $M$ (not necessarily $\M$-free), the algebraic loop $\Omega M$ (defined below in Definition \ref{de:omega}) is free over $\M$, as it is the kernel of a map from a finitely generated free $\M$-module.

The stable category $\st(A)$ is naturally enriched over $A$-modules, since the equivalence relation on morphisms is $A$-linear. 
The category $\st(A)$ has additional structure that we describe next.

\begin{pro}
\label{pro:stA-symm-monoidal}
The category $\st(A)$ is a closed symmetric monoidal category.
\end{pro}

\begin{proof}
This is a standard result from 
the theory of stable modules; see \cite[Proposition 15.2.19]{Mar83} for example.
The only additional observation is that 
the tensor product of $\M$-free modules is $\M$-free. 
\end{proof}

\subsection{Picard groups}

\begin{de}
\label{defn:Pic}
Let $A$ be a finite motivic Hopf algebra.
The Picard group $\Pic(A)$ is the group (of isomorphism classes) of invertible objects of $\st(A)$ under the monoidal structure, \emph{i.e.,} the group of stably invertible modules with the tensor product as group law. 
\end{de}

Note that $\Pic(A)$ is an abelian group because $\st(A)$ is symmetric monoidal.

\begin{rk}
\label{rk:Mfg}
In Definition \ref{defn:Pic},
we are only considering finitely generated $A$-modules.
This is no loss of generality because every invertible
object must be finitely generated.  This follows from
\cite[Proposition 2.1.3]{MS14}, for example.
\end{rk}

\begin{rk}
\label{rk:Mfree}
In Definition \ref{defn:Pic}, we have 
defined the Picard group using only $A$-modules that are
$\M$-free.  In fact, if $M$ and $N$ are arbitrary finitely generated
$A$-modules such that $M \otimes N$ is stably equivalent to $\M$,
then $M$ and $N$ must in fact be $\M$-free.
In other words, there is no harm in considering only $\M$-free 
modules in the Picard group.
For if $M \otimes N$ is isomorphic to $\M \oplus P$ for some
projective $A$-module $P$, then
$P$ is $\M$-free by Lemma \ref{lem:Aproj-Mfree}.
Therefore, $M \otimes N$ is $\M$-free, and $M$ and
$N$ are $\M$-free as well.
\end{rk}

\begin{de}
Denote the $\M$-linear dual functor by
$$D : \Amod^{\op} \lto \Amod \colon M \lmapsto DM = \Hom_{\M}(M, \M).$$
\end{de}

\begin{lemma}
\label{lem:stabdual}
The $\M$-linear dual functor $D$ induces a functor
$$D : \st(A)^{\op} \lto \st(A).$$
\end{lemma}

\begin{proof}
The dual functor $D$ 
preserves $\M$-freeness because $D$ is defined as $\Hom$ over $\M$.

It suffices to check that if $P$ is $A$-projective, 
then $DP$ is $A$-projective.
Since the dual respects direct sums, it is enough to show
that $DA$ is projective.
This follows as in \cite[Theorem 12.2.9]{Mar83} by considering
a retraction
\[
DA \lto DA \otimes A \lto DA
\]
and observing that the ``shearing map" 
\cite[Proposition 12.1.4]{Mar83}
makes $DA \otimes A$
into a free $A$-module.
\end{proof}

Lemma \ref{lemma:picarddual} shows that the dual functor $D$ corresponds
to inversion in the Picard group.

\begin{lemma} \label{lemma:picarddual}
Let $M$ be an $A$-module.
The evaluation morphism $DM \otimes M \stackrel{ev}{\lto} \M$ is a stable equivalence if and only if $M$ is invertible. 
In particular, 
the inverse of any element $[M]$ in $\Pic(A)$ is its dual $[DM]$. 

\begin{proof}
This fact is standard in stable module theory; see \cite[Proposition A.2.8]{HPS}. 
\end{proof}
\end{lemma}

We next describe the ``algebraic loop'' functor that is part
of the structure of a stable module category.

\begin{de} \label{de:omega}
Let $\Omega$ be the endo-functor of $\st(A)$ given by 
$$\Omega M = \ker( P \lto M)$$ 
where $P \lto M$ is any projective cover of $M$. 

For $k \geq 0$, define
$\Omega^k M$ inductively to be $\Omega (\Omega^{k-1} M)$.

For $k < 0$, define $\Omega^k M$ to be $D(\Omega^{-k} D M)$.
\end{de}

An immediate application of Schanuel's lemma 
shows that $\Omega M$ 
is independent of the choice of $P$.
Note that $\Omega M$ is $\M$-free because it is a subobject 
of $P$, and $P$ is $\M$-free by Lemma \ref{lem:Aproj-Mfree}.

\begin{lemma}
\label{lem:loop}
If $M$ is stably invertible, then so is $\Omega M$.
\end{lemma}

\begin{proof}
This is a standard part of the theory of stable modules. It follows from the fact that $\Omega M \cong \Omega \M \otimes M$, and that $\Omega \M$ is stably invertible; see \cite[Proposition 2.10]{BrOssa} for example.
\end{proof}

Lemma \ref{lem:loop} implies that there is a group homomorphism
$$\eta : \Z^3 \lto \Pic(A)$$
sending $(m,n,s)$ to the stable class of $ \Sigma^{m,n} \Omega^s \M$.
This homomorphism constructs
many elements in the Picard group of $A$.
Such elements exist for essentially formal reasons and do not really
reflect on the structure of the underlying algebra $A$.
In a sense, the image of $\eta$ consists of ``uninteresting"
invertible elements.

\section{$\tau$ quotients}

Suppose that $A$ is a finite motivic Hopf algebra. 
Then
$A / \tau = \F \otimes_{\M} A$ is a Hopf algebra.
Since $A / \tau$ is defined over a field $\F$, it is generally
easier to understand than $A$ itself.
We shall use a change of basis functor that relates our finite motivic Hopf algebra
$A$ to the Hopf algebra $A /\tau$.

\begin{pro} \label{pro:tauquotientexact}
Tensoring with the $\M$-module $\F$ induces a strongly monoidal functor
\begin{equation*}
\Amodf \stackrel{(-)/\tau}{\ltooo}  \Atmod^{\textnormal{f}}
\end{equation*}
that preserves exact sequences.
This functor passes to the stable category and thus induces a strongly monoidal functor
$$\st(A) \stackrel{(-)/\tau}{\ltooo}  \st(\At).$$

\begin{proof}
The unit $\M$ of the monoidal structure of $\Amodf$ is sent to the unit $\F$.   The functor is strongly monoidal since
\begin{equation*}
\quotient{M}{\tau} \otimes \quotient{N}{\tau} \cong \quotient{M \otimes N}{\tau};
\end{equation*}
this is just an application of commuting colimits.

Consider a short exact sequence in $\Amodf$.
The sequence is split exact on the underlying free $\M$-modules. 
It is still split exact as a sequence of 
$\F$-modules after tensoring with $\F$.
This shows that $(-)/\tau$ is exact.

The functor sends free $A$-modules to free $A/\tau$-modules. 
By additivity, we conclude that it sends projective $A$-modules to 
projective $A/\tau$-modules and thus descends to the stable categories.
\end{proof}
\end{pro}

\begin{rk}
Note that
$\Atmod^{\textnormal{f}} = \Atmod$, 
since the ground ring $\F$ is a field.
\end{rk}

\begin{rk}
The functor
$\Amodf \stackrel{(-)/\tau}{\ltooo}  \Atmod^{\textnormal{f}}$
of Proposition \ref{pro:tauquotientexact} preserves exact sequences,
but it is not an ``exact functor" in the usual sense because
$\Amodf$ is not an abelian category.  Namely, the cokernel of a map
in $\Amodf$ need not be $\M$-free.
\end{rk}

We now come to the first major result that will allow us
to understand the stable module category of a finite motivic Hopf algebra $A$. 
Lemma \ref{lemma:killtau} identifies projective $A$-modules
in terms of their quotients by $\tau$.

\begin{lemma} \label{lemma:killtau}
Let $A$ be a finite motivic Hopf algebra, 
and let $M$ be a finitely generated $A$-module that is $\M$-free.
The following conditions are equivalent:
\begin{enumerate}
\item $M$ is projective as an $A$-module.
\item $M/\tau$ is projective as an $A/\tau$-module.
\item $M/\tau$ is free as an $A/\tau$-module.
\end{enumerate}

\begin{proof}
Note that $A/\tau$ is a Frobenius algebra since it is a finite
dimensional Hopf algebra over the field $\F$ \cite[Theorem 12.2.9]{Mar83}.
In particular,
projective $A/\tau$-modules and free $A/\tau$-modules are the same.
This shows that conditions (2) and (3) are equivalent.

Now suppose that $M$ is a projective $A$-module. 
Then $M/\tau$ is a projective $A/\tau$-module by Proposition \ref{pro:tauquotientexact}. 
This shows that condition (1) implies condition (2).

To show that condition (3) implies condition (1), 
suppose that $M/\tau$
 is a free $A/\tau$-module.
We will show that $\ext^i_A(M,N)$ vanishes for all $A$-modules $N$.
In fact, it suffices to assume that $N$ is finitely generated,
for $\Hom_A(M, -)$ commutes with filtered colimits since
$M$ is finitely generated, and filtered colimits are exact.

Since $M/\tau$ is free over $A/\tau$, we have an
$A$-free resolution of $M / \tau$ of the form
\begin{equation*}
\cdots \lto 0 \lto \bigoplus A \stackrel{\tau}{\lto} \bigoplus A \lto M/\tau.
\end{equation*}
Therefore,
$\ext^{i}_A (M/\tau,N)$ vanishes whenever $i \geq 2$ and $N$ is any
$A$-module.

Since $M$ is $\M$-free, we have a short exact sequence
\begin{equation*}
0 \lto M \lto M \lto M/\tau \lto 0.
\end{equation*}
This sequence induces a long exact sequence 
\begin{equation*}
\cdots \lto \ext^{i}_A (M,N) \stackrel{\tau}{\lto} 
\ext^{i}_A (M,N) \lto \ext^{i+1}_A (M/\tau,N)
\lto \cdots
\end{equation*}
for all $A$-modules $N$.
Since $\ext^{i+1}_A (M/\tau,N)$ is zero for $i \geq 1$ by the previous 
paragraph, we conclude that
the map
\begin{equation*}
\ext^{i}_A(M,N) \stackrel{\tau}{\lto} \ext^{i}_A(M,N) 
\end{equation*}
is surjective for $i \geq 1$. 

Note that $M$ and $N$ are finitely generated as $\M$-modules since
they are finitely generated as $A$-modules, and $A$ is a finitely
generated $\M$-module.
This implies that $\ext^{i}(M,N)$ vanishes in sufficiently low motivic
weights. 
The surjectivity of multiplication by $\tau$ then implies that
$\ext^i (M,N)$ vanishes in all weights.
This means that $M$ is a projective $A$-module.
\end{proof}
\end{lemma}

\begin{lemma}
\label{lem:f/tau}
Let $M$ and $N$ be finitely generated $A$-modules that are also
$\M$-free, and let $f: M \lto N$ be a map such that
$f/\tau: M/\tau \lto N/\tau$ is injective.  Then
$f$ is also injective, and the cokernel of $f$ is $\M$-free.
\end{lemma}

\begin{proof}
Suppose that $x$ is an element of $M$ such that $f(x) = 0$,
and let $\overline{x}$ be the corresponding element in $M/\tau$.
Then $(f/\tau) (\overline{x})$ is zero,
so $\overline{x}$ is also zero because $f/\tau$ is injective.
Therefore, $x$ equals $\tau y$ for some $y$.
Now $\tau f(y) = f(\tau y) = f(x) = 0$,
so $f(y)$ is also zero since $N$ is $\M$-free.

This shows that the kernel of $f$ consists of elements that are 
infinitely divisible by $\tau$.  Since $M$ is finitely generated,
the kernel must be zero.

Now consider the cokernel $N/M$ of $f$.  Since $N/M$ is finitely generated,
it suffices to consider the annihilator of $\tau$
in $N/M$.  We will show that this annihilator is zero.

Let $x$ be an element of $N$, and let $\overline{x}$ be the element
of $N/M$ that it represents.  Suppose that $\tau \overline{x}$ is zero.
Then $\tau x$ belongs to $M$.  Since $f/\tau$ is injective
and $(f/\tau) (\tau x)$ is zero, we conclude as in the first paragraph
that $\tau x$ equals $\tau y$ for some $y$ in $M$.
Since $N$ is $\M$-free, it follows that $x$ equals $y$.
In particular, $x$ belongs to $M$.  In other words,
$\overline{x}$ is zero.
\end{proof}

The strong monoidal exact functor
$$-/\tau : \st(A) \lto \st(A/\tau)$$
of Proposition \ref{pro:tauquotientexact}
induces a group homomorphism
$$V : \Pic(A) \lto \Pic(A/\tau).$$

\begin{pro} \label{pro:picardmaps}
The map $V : \Pic(A) \lto \Pic(A/\tau)$
is injective.
\begin{proof}
Let $M$ be a finitely generated $A$-module such that $M$ is $\M$-free,
and suppose that $[M]$ in $\Pic(A)$ belongs to the kernel of $V$.
Equivalently, $M/\tau$ and $\F$ are stably equivalent $A/\tau$-modules.
Since $A/\tau$ is a finite dimensional Frobenius algebra over $\F$,
we can use \cite[Proposition 14.11]{Mar83} to see that
$M/\tau$ is isomorphic to the direct sum of $\F$ and a free
$A/\tau$-module.  In other words,
$M/\tau$ is isomorphic to $\F \oplus F/\tau$, where
$F$ is a free $A$-module.
Let $j$ be the injection $F/\tau \lto M/\tau$.

There is a commutative diagram
\begin{equation*} \label{eq:proofofpicardmap}
\xymatrix{ M \ar@{->>}[r] & M/\tau \\
F \ar@{-->}[u]^i \ar@{->>}[r] & F/\tau, \ar@{^(->}[u]_j }
\end{equation*}
in which the dashed arrow exists because
$F$ is $A$-projective and $M \lto M/\tau$ is a surjection.
By Lemma \ref{lem:f/tau},
$i$ is injective because $j$ is injective.

We now compute the cokernel $C$ of $i$.
Lemma \ref{lem:f/tau} implies that $C$ is $\M$-free.
Then Proposition \ref{pro:tauquotientexact} says that
$C/\tau$ is isomorphic to the cokernel of $j$,
which is $\F$ by inspection.
We conclude that $C$ is isomorphic to $\M$.

Thus, there is a short exact sequence
\begin{equation*}
F \inj M \surj \M,
\end{equation*}
so $M \surj \M$ is a stable equivalence and $[M]$ is trivial in $\Pic(A)$.
\end{proof}
\end{pro}

\section{The finite motivic Hopf algebra $\A(1)$}

In this section, we introduce the specific 
finite motivic Hopf algebra $\A(1)$ whose Picard group we will compute.

\begin{de} \label{de:a1}
The finite motivic Hopf algebra $\A(1)$ is the 
$\M$-subalgebra of the motivic Steenrod algebra generated by $\Sq^1$ and $\Sq^2$.
\end{de}

\begin{lemma}
\label{lem:A(1)}
The finite motivic Hopf algebra $\A(1)$ is isomorphic to 
$$\frac{\M[\Sq^1,\Sq^2]}{\Sq^1 \Sq^1, \Sq^2\Sq^2 + \tau \Sq^1\Sq^2\Sq^1, \Sq^1\Sq^2\Sq^1\Sq^2 + \Sq^2\Sq^1\Sq^2\Sq^1}.$$
The element $\Sq^1$ is primitive, and 
$\Delta(\Sq^2) = \Sq^2 \otimes 1 + \tau \Sq^1 \otimes \Sq^1 + 1 \otimes \Sq^2$. 
\end{lemma}

\begin{proof}
This follows immediately from Voevodsky's description of the 
motivic Steenrod algebra \cite{Voe03}.
\end{proof}

See Figure \ref{fig:a1} for a picture of $\A(1)$.
When writing $\A(1)$-modules we use the following conventions.
A straight line represents the action of $\Sq^1$, a curved line represents the action of $\Sq^2$, and a line is dotted if a squaring operation 
hits $\tau$ times a generator. For example, the dotted line in 
Figure \ref{fig:a1} shows the relation $\Sq^2\Sq^2 = \tau \Sq^1\Sq^2\Sq^1.$ 

\begin{figure} 
\begin{tikzpicture}[line cap=round,line join=round,>=triangle 45,x=1.0cm,y=1.0cm]

\filldraw [black] (0,0) circle (0.07)
                  (0,1) circle (0.07)
                  (0,2) circle (0.07)
                  (-1,3) circle (0.07)
                  (1,3) circle (0.07)
                  (0,4) circle (0.07)
                  (0,5) circle (0.07)
                  (0,6) circle (0.07);
\draw (0,0) -- (0,1)
      (0,2) -- (1,3)
      (-1,3) -- (0,4)
      (0,5) -- (0,6);
\draw (0,0) .. controls (0.7,0.7) and (0.7,1.3) .. (0,2)
      (0,1) .. controls (-0.7,1.7) and (-1,2.3) .. (-1,3)
      (1,3) .. controls (1,3.7) and (0.7,4.3) .. (0,5)
      (0,4) .. controls (-0.7,4.7) and (-0.7,5.3) .. (0,6);
\draw [dashed] (0,2) .. controls (0.7,2.7) and (0.7,3.3) .. (0,4);
\draw (-0.3,6) node {$1$};
\draw (-0.45,4) node {$\Sq^2$};
\draw (0,5) node[anchor=south west] {$\Sq^1$};

\end{tikzpicture}
\caption{The finite motivic Hopf algebra $\A(1)$}
\label{fig:a1}
\end{figure}
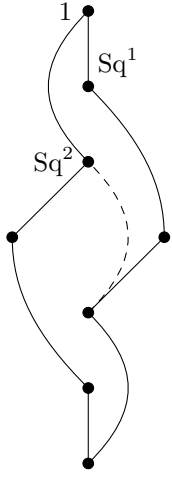

\begin{lemma}
\label{lem:A(1)/tau}
As ungraded Hopf algebras,
$\A(1)/\tau$ 
is isomorphic to the group algebra $\F[D_8]$ of the dihedral group $D_8$
of order $8$.
\begin{proof}
Lemma \ref{lem:A(1)} implies that $\A(1) / \tau$ is isomorphic to
$$\frac{\F[\Sq^1,\Sq^2]}{\Sq^1 \Sq^1, \Sq^2\Sq^2, 
\Sq^1\Sq^2\Sq^1\Sq^2 + \Sq^2\Sq^1\Sq^2\Sq^1}.$$
For our purposes, a convenient presentation of $D_8$ consists of
two generators $x$ and $y$ with the relations
$x^2$, $y^2$, and $(xy)^4$.
The isomorphism from $\A(1)/\tau$ to $\F [D_8]$
takes $\Sq^1$ to $1 +x$ and $\Sq^2$ to $1+y$.
\end{proof}
\end{lemma}

Recall that a sub-Hopf algebra $B$ of a Hopf $\F$-algebra $A$ is elementary if it is isomorphic to an exterior algebra.
Note that $Q_0 = \Sq^1$ and $Q_1 = \Sq^2 \Sq^1 + \Sq^1 \Sq^2$
are elements of $\A(1)$ whose squares are zero.

\begin{lemma}
\label{lem:elementary}
The maximal elementary sub-Hopf algebras of
$\A(1)/\tau$ are the exterior algebras
$E(Q_0, Q_1)$ and $E(\Sq^2, Q_1)$.
\end{lemma}

\begin{proof}
Lemma \ref{lem:A(1)/tau} says that $A/\tau$ is isomorphic to the 
group algebra $\F [D_8]$ of the dihedral group of order 8.
The elementary sub-Hopf algebras of $\F[D_8]$ correspond
to the elementary abelian $2$-subgroups of $D_8$.
The group $D_8$ has two maximal elementary abelian subgroups.
Tracing back through the isomorphism of Lemma \ref{lem:A(1)/tau},
one can identify the two maximal elementary
sub-Hopf algebras of $\A(1)/\tau$.
\end{proof}

\subsection{Margolis homology}

We now turn to an algebraic invariant detecting projectivity of $\A(1)$-modules, analogous to Margolis's techniques  using 
$P^s_t$-homology \cite{Mar83}. 

\begin{de}
Let $x$ be an element of $A$ such that $x^2$ is zero.
For any $A$-module $M$, define the
Margolis homology $H(M;x)$ to be the
annihilator of $x$ modulo the submodule $x M$.
\end{de}

Recall that classically,
an $\A(1)^{\cl}$-module $M$ is projective if and only if
$H(M; Q_0)$ and $H(M;Q_1)$ are both zero \cite[Theorem 3.1]{Ad71}, which is a direct consequence of a more general result \cite[Theorem 1.2-1.4]{Pa97}.
Our goal is to generalize this result to the motivic situation.

Unfortunately, the motivic situation is more complicated.
If $M$ is an $\A(1)$-module and
$H(M; Q_0)$ and $H(M;Q_1)$ both vanish,
then $M$ is not necessarily projective.

\begin{ex}
Let $\widetilde{\A}(1)$ be the $\A(1)$-module on two generators
$x$ and $y$ of degrees $(0,0)$ and $(2,0)$ respectively,
subject to the relations
$\Sq^2 x = \tau y$ and
$\Sq^1 \Sq^2 \Sq^1 x = \Sq^2 y$.
Figure \ref{fig:a1tilda} represents $\widetilde{\A}(1)$ as
an $\A(1)$-module.

The Margolis homology groups
$H(\widetilde{\A}(1); Q_0)$ and
$H(\widetilde{\A}(1); Q_1)$ both vanish.
However, $\widetilde{\A}(1)$ is not a projective $\A(1)$-module.
\end{ex}

\begin{figure} 
\begin{tikzpicture}[line cap=round,line join=round,>=triangle 45,x=1.0cm,y=1.0cm]

\filldraw [black] (0,0) circle (0.07)
                  (0,1) circle (0.07)
                  (0,2) circle (0.07)
                  (-1,3) circle (0.07)
                  (1,3) circle (0.07)
                  (0,4) circle (0.07)
                  (0,5) circle (0.07)
                  (0,6) circle (0.07);
\draw (0,0) -- (0,1)
      (0,2) -- (1,3)
      (-1,3) -- (0,4)
      (0,5) -- (0,6);
\draw 
      (0,1) .. controls (-0.7,1.7) and (-1,2.3) .. (-1,3)
      (0,2) .. controls (0.7,2.7) and (0.7,3.3) .. (0,4)
      (1,3) .. controls (1,3.7) and (0.7,4.3) .. (0,5);
\draw [dashed] (0,0) .. controls (0.7,0.7) and (0.7,1.3) .. (0,2)
               (0,4) .. controls (-0.7,4.7) and (-0.7,5.3) .. (0,6);
\draw (0,6) node[anchor=east] {$x$};
\draw (-0.3,4) node {$y$};
               
\end{tikzpicture}
\caption{The $\A(1)$-module $\widetilde{\A}(1)$}
\label{fig:a1tilda}
\end{figure}

It turns out that we need two additional criteria for
projectivity beyond
$Q_0$-homology and Margolis $Q_1$-homology.

\begin{pro} \label{pro:margolisdetection}
Let $M$ be a finitely generated $\A(1)$-module. 
Then $M$ is projective if and only if:
\begin{enumerate}
\item
$M$ is free over $\M$; and
\item
$H(M/\tau;Q_0)=0$; and
\item
$H(M/\tau;Q_1)=0$; and
\item
$H(M/\tau;\Sq^2)=0$.
\end{enumerate}

\begin{proof}
First suppose that $M$ is projective.
By inspection, conditions (2) through (4)
are satisfied when $M$ is $\A(1)$.  Therefore,
these conditions are satisfied when $M$ is free.  Using
that a projective module is a summand of a free module,
conditions (2) through (4) are also satisfied for any projective $M$.
Finally, Lemma \ref{lem:Aproj-Mfree} shows that condition (1) is satisfied.

Now suppose that conditions (1) through (4) are satisfied.
By Lemma \ref{lemma:killtau}, it suffices to show
that $M/\tau$ is $A/\tau$-projective.
Note that $A/\tau$-projectivity is detected by restriction
to the quasi-elementary sub-Hopf algebras of $A/\tau$ \cite[Theorem 1.2-1.4]{Pa97}.
See \cite[Definition 1.1]{Pa97} for the definition of quasi-elementary sub-Hopf algebras.

For group algebras, quasi-elementary sub-Hopf algebras coincide
with elementary sub-Hopf algebras \cite{Ser65} (as observed in \cite{Pa97}).
Since $A/\tau$ is isomorphic to the group algebra $\F [D_8]$
by Lemma \ref{lem:A(1)/tau},
Lemma \ref{lem:elementary} shows that
the quasi-elementary sub-Hopf algebras of $A/\tau$ 
are the exterior algebra $E(Q_0, Q_1)$ and the
exterior algebra $E(\Sq^2, Q_1)$.
Conditions (2) and (3) imply that
$M/\tau$ is $E(Q_0, Q_1)$-projective,
and conditions (3) and (4) imply that
$M/\tau$ is $E(\Sq^2, Q_1)$-projective.
\end{proof}
\end{pro}

\begin{rk}
The exterior algebra $E(Q_0, Q_1)$ is the unique
maximal quasi-elem\-ent\-ary sub-Hopf algebra of the 
classial Hopf algebra $\A(1)^{\cl}$.
This explains why condition (4) of Proposition \ref{pro:margolisdetection}
is absent from the classification of projective
$\A(1)^{\cl}$-modules.
\end{rk}

\begin{cor}
\label{cor:margolisdetection}
Let $M$ and $N$ be finitely generated $\A(1)$-modules 
that are $\M$-free, and let
$f: M \lto N$ be an $\A(1)$-module map.
Then $f$ is a stable equivalence if and only if 
$f/\tau: M/\tau \lto N/\tau$
induces an isomorphism
in Margolis homologies with respect to $Q_0$, $Q_1$, and $\Sq^2$.
\end{cor}

\begin{proof}
We may choose a free $\A(1)$-module $F$ and a surjective map
$g: M \oplus F \lto N$ that restricts to $f$ on $M$.
Then $f$ is a stable equivalence if and only if $g$ is a stable
equivalence,
and $f/\tau$ induces isomorphisms in Margolis homologies
if and only if $g/\tau$ induces isomorphisms in Margolis
homologies.
In other words, we may assume that $f$ is surjective.
(From a model categorical perspective, we have replaced $f$
by an equivalent fibration.)

Let $K$ be the kernel of $f$.
The short exact sequence 
\[
0 \lto K \lto M \stackrel{f}{\lto} N \lto 0
\]
induces a short exact sequence
\[
0 \lto K/\tau \lto M/\tau \stackrel{f/\tau}{\lto} N/\tau \lto 0
\]
by Proposition \ref{pro:tauquotientexact}.
This last short exact sequence induces long exact sequences
in Margolis homologies with respect to $Q_0$, $Q_1$ and $\Sq^2$.
The long exact sequence shows that
$f/\tau$ is an isomorphism in Margolis homologies if and only if
$K/\tau$ has vanishing Margolis homologies.
Finally,
Proposition \ref{pro:margolisdetection}
implies that $K/\tau$ has vanishing Margolis homologies
if and only if $K$ is projective.
Note that $K$ is finitely generated and $\M$-free because it is a 
subobject of the finitely generated $\M$-free module $M$.
Finally, $K$ is projective if and only if
$f$ is a stable equivalence.
\end{proof}

We establish a K\"unneth theorem for Margolis homology.

\begin{pro} \label{pro:kunneth}
Let $M$ and $N$ be $\A(1)$-modules that are free over $\M$. Then
$$H(M/\tau \otimes N/\tau;x) \eq H(M/\tau;x) \otimes H(N/\tau;x)$$
when $x$ is $Q_0$, $Q_1$, or $\Sq^2$.
\begin{proof}
Lemma \ref{lem:A(1)} gives
the coproduct formula
\[
\Delta (\Sq^2) = 
\Sq^2 \otimes 1 + \tau \Sq^1 \otimes \Sq^1 + 1 \otimes \Sq^2.
\]
Therefore, $\Sq^2$ is primitive modulo $\tau$.
In particular, it acts as a derivation on
$M/\tau \otimes N/\tau$.
The isomorphism in $\Sq^2$-homology follows from the
classical K\"unneth formula for chain complexes over $\F$.

The arguments for $Q_0$ and $Q_1$ are the same, except slightly easier
because these elements are primitive even before quotienting by $\tau$.
\end{proof}
\end{pro}

\begin{pro}
\label{pro:Margolis-invertible}
Let $M$ be a finitely generated $\A(1)$-module that is $\M$-free.
Then $M$ is invertible if and only if
$M/\tau$ has one-dimensional Margolis homologies with respect to 
$Q_0$, $Q_1$, and $\Sq^2$.

\begin{proof}
First suppose that $M$ is invertible.  In other words,
there exists an $\A(1)$-module $N$ and 
a stable equivalence
$$ M \otimes N \stackrel{\steq}{\lto} \M.$$
Proposition \ref{pro:tauquotientexact} implies that there is
a stable equivalence
$$ (M \otimes N)/\tau \stackrel{\steq}{\lto} \F$$
of $\A(1)/\tau$-modules.
Corollary \ref{cor:margolisdetection} shows that
$$ H ((M \otimes N) / \tau; x) \lto H(\F; x)$$
is an isomorphism when $x$ is $Q_0$, $Q_1$, or $\Sq^2$.
Now use Proposition \ref{pro:kunneth} to deduce that
$H(M/\tau;x) \otimes H(N/\tau;x)$ is isomorphic to $\F$.
It follows that
$H(M/\tau;x)$ is one-dimensional.

Now assume that $M/\tau$ has one-dimensional Margolis homologies.
Note that
\[
H(D(M/\tau);x) \eq \Hom_{\F}(H(M/\tau;x);\F)
\]
when $x$ is $Q_0$, $Q_1$, or $\Sq^2$.
Therefore,
$D(M/\tau)$ also has one-dimensional Margolis homologies.
By Proposition \ref{pro:kunneth},
$M/\tau \otimes D(M/\tau)$ also has one-dimensional Margolis homologies.
Hence the evaluation map
$$ M/\tau \otimes D(M/\tau) \lto \F$$
induces an isomorphism in Margolis homologies because both sides
are one-dim\-ension\-al.
Note that $M/\tau \otimes D(M/\tau)$ is isomorphic to
$(M \otimes DM)/\tau$ by Proposition \ref{pro:tauquotientexact}.
Finally, Corollary \ref{cor:margolisdetection} shows
that the evaluation map
$$ M \otimes DM \lto \M$$
is a stable equivalence.
This shows that $M$ is invertible with inverse $DM$.
\end{proof}
\end{pro}

\section{The Picard group of $\A(1)$}

\begin{de} \label{de:joker}
Let $J$ be the $\A(1)$-module on two generators $x$ and $y$
of degrees $(0,0)$ and $(2,0)$ respectively,
subject to the relations
$\Sq^2 x = \tau y$,
$\Sq^1 \Sq^2 \Sq^1 x = \Sq^2 y$, and $\Sq^1 y = 0$.
\end{de}

Figure \ref{fig:joker} represents $J$ as an $\A(1)$-module.

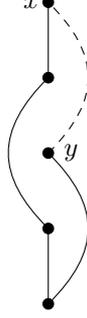
\begin{figure} 
\begin{tikzpicture}[line cap=round,line join=round,>=triangle 45,x=1.0cm,y=1.0cm]

\filldraw [black] (0,0) circle (0.07)
                  (0,1) circle (0.07)
                  (0,2) circle (0.07)
                  (0,3) circle (0.07)
                  (0,4) circle (0.07);
\draw (0,0) -- (0,1)
      (0,3) -- (0,4);
\draw 
      (0,0) .. controls (0.7,0.7) and (0.7,1.3) .. (0,2)
      (0,1) .. controls (-0.7,1.7) and (-0.7,2.3) .. (0,3);
\draw [dashed] (0,2) .. controls (0.7,2.7) and (0.7,3.3) .. (0,4);
\draw (0.3,2) node {$y$};
\draw (0,4) node[anchor=east] {$x$};
\end{tikzpicture}
\caption{The $\A(1)$-module $J$}
\label{fig:joker}
\end{figure}

\begin{lemma}
The $\A(1)$-module $J$ is invertible,
and the order of $[J]$ in $\Pic(\A(1))$ is infinite.
\end{lemma}

\begin{proof}
Proposition \ref{pro:Margolis-invertible} implies that $J$ is invertible.
The $Q_0$-homology and $Q_1$-homology of $J/\tau$ are generated by $x$, 
while the $\Sq^2$-homology of $J/\tau$ is generated by $y$.

The degrees of $x$ and $y$ are different.
Therefore, the $\Sq^2$-homology and the $Q_0$-homology of any tensor power
$J^{\otimes n}$ of $J$ are in different degrees.
On the other hand, the $\Sq^2$-homology and the $Q_0$-homology of
$\M$ are in the same degree.
This shows that $J^{\otimes n}$ is not stably equivalent to $\M$.
\end{proof}

\begin{rk}
The classical joker is self-dual as an $\A(1)^{\cl}$-module.
Therefore, it represents an element of order two in
$\Pic(\A(1)^{\cl}$. 
On the other hand, Figure \ref{fig:joker} shows that
the motivic joker is not self-dual.
\end{rk}

\begin{thm}
There is an isomorphism 
$$ \Z^4 \lto \Pic(\A(1))$$
sending $(a,b,c,d)$ to the class of
$\Sigma^{a,b} \Omega^c J^d$.
\begin{proof}
Recall the homomorphism
$$V : \Pic(\A(1)) \lto \Pic(A/\tau)$$
from Proposition \ref{pro:picardmaps}.
Consider the composition
$$\Z^4 \lto \Pic(\A(1)) \stackrel{V}{\lto} \Pic(\A(1)/\tau)
\stackrel{\cong}{\lto} \Pic(\F [D_8]),$$
where the last isomorphism comes from Lemma \ref{lem:A(1)/tau}.

Recall from \cite[Theorem 5.4]{CT00} that the ungraded Picard group of 
$\F [D_8]$ is isomorphic to 
$\Z^2$, generated by $\Omega \F$ and a module $L$. 
If we add the motivic bigrading, then we obtain that
the graded Picard group $\Pic(\F [D_8])$ is isomorphic to $\Z^4$.

By direct computation, the composition sends the 
joker $J$ to $\Omega L$.  Therefore,
the composition is an isomorphism.  This shows that $V$ is surjective.
We already know that $V$ is injective
from Proposition \ref{pro:picardmaps}.
Therefore, $V$ is an isomorphism, so the map
$$ \Z^4 \lto \Pic(\A(1))$$ 
is an isomorphism as well.
\end{proof}
\end{thm}

\bibliographystyle{alpha}
\bibliography{biblio}

\end{document}